\documentclass{article}
\usepackage{amsmath, amssymb}

\newtheorem{theorem}{Theorem}
\newtheorem{proposition}{Proposition}
\newtheorem{lemma}[theorem]{Lemma}
\newtheorem{definition}{Definition}
\newtheorem{example}{Example}

\title{On some arithmetic conditions of recurrent sequences modulo prime $p$ \footnote{This paper is an output of a research project implemented as part of the Basic Research Program at HSE University.}}

\author{Ilya Vyugin, Sashadhar Dutta}

\begin{document}

\maketitle

\begin{abstract}
We study the $K$-Fibonacci sequence $\mathcal{F}_p$ modulo prime $p$. Cardinalities of sets $|\mathcal{F}_p+\mathcal{F}_p|$ and $|\mathcal{F}_p\cdot\mathcal{F}_p|$ are estimated. We present the method of estimating doubling constant of some $m$-dimensional recurrent sets in $\mathbb{F}_p$.
%study some properties of $m$-dimensional recurrent sequences modulo $p$. 
\end{abstract}

\section{Introdution}

Fibonacci sequence 
$$
F_n = F_{n-1} + F_{n-2} \quad \text{with}\quad F_0 = 0, \, F_1 = 1
$$
is the most famous recurrent sequence. It was established in several different parts of the world. In 400 BC famous Indian poet Pingala wrote about Fibonacci sequence (see \cite{PS}). And also previously in Hindu scriptures, holy book Visnudharmottara Purana, in one chapter it was written a great knowledge on Fibonacci Sequence. %( ref. https://www.cs.umd.edu/~gasarch/BLOGPAPERS/fibfibs.pdf). I am refring these two articles for more details.
%{\bf Short history.}
In Europe it was first established by Fibonacci in 1202. Fibonacci sequence have many important applications in biology, poetry and social life.

%(https://www.cs.umd.edu/~gasarch/BLOGPAPERS/fibfibs.pdf)

We study a generalization of the Fibonacci sequence (see \cite{B-L})
\begin{eqnarray}\label{K-Fib}
F_{n+2}=K F_n+F_{n+1},\qquad F_0=0,\quad F_1=1  
\end{eqnarray}
modulo prime $p$. It is called the $K$-Fibonacci sequence ($K\in\mathbb{F}_p^*$). Note that if $K=1$ then the  sequence (\ref{K-Fib}) is Fibonacci sequence, and if $K=2$ then (\ref{K-Fib}) is a Pell number sequence. Methods of the paper can be applied to a more general class of modulo prime recurrent sequences, including some recurrent sequences of degree $m>2$. To apply the method to a given recurrent sequence, it is necessary to prove the irreducibility of the series of polynomials. We present these calculations for the case of $K$-Fibonacci sequences using Newton polygons technique.

%$K$-Fibonacci period:
% the length of this period, often termed the "Pisano period"
% $\pi_K (p)$, is  $K^2 + 4$ by the factorization of
%and the modulus. Periods are generally smaller for larger m
%and have 1, 2, or 4 zeros (see \cite{B-L}).

%Even Modulo: Every 3rd Fibonacci number is even ($F_3,F_6,F_9$), meaning even moduli have shorter more regular periods because they interact directly with this 3-cycle.
%Odd Modulo 

An important question in the arithmetic number theory is what happens to a set $A$ when it is doubled? This can be additive doubling:
$$
A+A=\{ a+b \mid a,b\in A\},
$$ 
or multiplicative doubling:
$$
A\cdot A=\{ a\cdot b \mid a,b\in A\}.
$$ 
Properties of doubling of a set depend on whether it has an additive or multiplicative structure. 
The sum-product problem posed by Szemerédi is well known:

\medskip

{\it For any $\varepsilon>0$ and for any sufficiently large set $A$ the following
$$
|A+A\cdot A|=\#\{ a+bc\mid a,b,c\in A \}>C|A|^{2-\varepsilon}.
$$
holds with some constant $C$.}

\medskip

We consider a set $\mathbf{X}$ of elements of a recurrent sequence of $m$ dimensions
\begin{eqnarray}\label{rec-rel}
X_{n+m} = \alpha_1X_{n+m-1} + \cdots + \alpha_mX_n,\quad \alpha_1\alpha_m\not=0    
\end{eqnarray}
starting from 
\begin{eqnarray}\label{rec-init}
X_1 = X_1^0, \cdots, X_m = X_m^0    
\end{eqnarray}
where the sequence $X_1,X_2\ldots$ and the coefficients $\alpha_1,\ldots,\alpha_m$ are considered modulo prime $p$. We assume that they are elements of the field $\mathbb{F}_p$. Since the sequence $\{ X_n\}_{n=1}^{\infty}$ is periodic, let us denote by $\mathbf{X}=\{X_n\mid n\in\mathbb{N}\}$ the set of elements of this sequence.
The length of the period of Fibonacci sequence modulo $p$ is called Pisano period $\pi(p)$. 
%We denote by $\pi_K(p)$ the period of $K$-Fibonacci sequence modulo $p$. 
One can read about the properties of $K$-Fibonacci numbers in~\cite{B-L}. Let us denote the set of $K$-Fibonacci numbers modulo prime $p$ by $\mathcal{F}_p$.
%The period of $K$-Fibonacci
%$\pi_K (p)$, is  $K^2 + 4$ by the factorization of
%and the modulus. Periods are generally smaller for larger m
%and have 1, 2, or 4 zeros (see \cite{B-L}).

Let us consider the polynomial:
\begin{eqnarray}\label{poly}
P(X,Y)=\sum_{i,j}a_{ij}X^iY^j.
\end{eqnarray}
We can study the polynomial doubling:
$$
P(\mathbf{X},\mathbf{X})=\{ P(X,Y) \mid X,Y\in\mathbf{X}\}
$$
of the set $\mathbf{X}$. Polynomial doubling generalizes additive and multiplicative doubling. In this paper, we study the cardinality of a polynomial doubling of the $K$-Fibonacci subset $\mathbf{X}$ of $\mathbb{F}_p$. The same problem for multiplicative subgroups is studied in \cite{AV}. The main result of the paper is the following theorem.

\begin{theorem}\label{th-est}
For $|\mathcal{F}_p|<\frac{1}{6}p^{3/4}$ the bounds
$$
|\mathcal{F}_p+\mathcal{F}_p|>C |\mathcal{F}_p|^{4/3},\qquad |\mathcal{F}_p\cdot\mathcal{F}_p|>C |\mathcal{F}_p|^{4/3}
$$
are holds with some absolute constant $C$.
\end{theorem}

\section{Background}

We study $m$-dimensional recurrent sequences (\ref{rec-rel}) with initial data (\ref{rec-init}) modulo prime $p$.
%\begin{eqnarray}\label{rec-rel}
%X_{n+m} = a_1X_{n+m-1} + \cdots + a_mX_n    
%\end{eqnarray}
%starting from 
%\begin{eqnarray}\label{rec-init}
%X_1 = X_1^0, \cdots, X_m = X_m^0    
%\end{eqnarray}
%where the sequence $X_1,X_2\ldots$ is considered modulo prime $p$.  The coefficients $a_1,\ldots,a_m$ and $X_n$, $n\in\mathbb{Z}$ belong to $\mathbb{F}_p$. Since the sequence $\{ X_n\}_{n=1}^{\infty}$ is periodic, let us denote by $\mathbf{X}=\{X_n\mid n\in\mathbb{N}\}$ the set of elements of this sequence. Let us denote by $|\mathbf{X}|=t$ the length of the period of $\{X_n\}$. 
A recurrent sequence is determined by initial data and by roots $\lambda_i$, $i=1,\ldots,m$ of its characteristic polynomial:
\begin{eqnarray}\label{char-poly}
\lambda^m - \alpha_1\lambda^{m-1} - \cdots - \alpha_{m-1}\lambda - \alpha_m = 0.
\end{eqnarray}
We assume that there is $\mu\in \overline{\mathbb{F}}_p$ such that 
$$
\lambda_i=\mu^{k_i},\quad k_i\in\mathbb{Z},\quad i=1,\ldots,m,
$$
all $\lambda_i$ are pairwise distinct and $k_1,\ldots,k_m$ have no common divisor.
%We call such a number of roots $\lambda_1,\ldots,\lambda_m$ {\it power dependent}.
%Our results occur under restricted $k_i$, $i=1,\ldots,m$. A primitive root always exists, but the powers can be larger and our bounds will be meaningless.

Elements of a recurrent sequence are expressed in the following form:
$$
X_n = \alpha_1\lambda_1^n + \cdots + \alpha_m\lambda_m^n,
$$
where 
$$
\left(\begin{array}{ccc}
\alpha_1 \\
\vdots \\
\alpha_m
\end{array}\right)=\left(\begin{array}{ccc}
1 & \ldots & 1 \\
\vdots & & \vdots \\
\lambda_1^{m-1} & \ldots & \lambda_n^{m-1}
\end{array}\right)^{-1}\left(\begin{array}{ccc}
X_1^0 \\
\vdots \\
X_m^0
\end{array}\right).
$$
In our assumptions, we have
$$
X_n = \alpha_1\mu^{nk_1} + \cdots + \alpha_m\mu^{nk_m}.
$$
Powers of the element $\mu$ generates the subgroup $G=\{ \mu^n \mid n\in\mathbb{Z}\}\subset\mathbb{F}_{p^s}^*$ of the algebraic extension of $\mathbb{F}_p$ by the characteristic polynomial (\ref{char-poly}) of order $|G|$ ($s\leqslant m$). Consequently, we have
$$
X_n=\alpha_1 x^{k_1}+\ldots+\alpha_m x^{k_m},\quad x=\mu^n\in G.
$$
The set $\mathbf{X}$ can be expressed as follows
$$
\mathbf{X}=\{ \alpha_1 x^{k_1}+\ldots+\alpha_m x^{k_m}\mid x\in G \}.
$$
Without loss of generality, we assume that $k_1\geqslant \ldots\geqslant k_m$.

\begin{lemma}
The following bounds
$$
|\mathbf{X}|\leqslant |G|,\qquad |G|\leqslant \max (k_1,k_1-k_m)|\mathbf{X}|.
$$
are satisfied.
\end{lemma}

{\it Proof.} The first bound is obvious. The given $X$ can be represented by polynomial
$$
(\alpha_1 x^{k_1}+\ldots+\alpha_m x^{k_m}-X)x^{\max(0,-k_m)}=0.
$$
For any $X$ there are not more than $\max (k_1-k_m,k_1)$ different $x\in\mathbb{F}_p$, because $\max(k_1,k_1-k_m)$ is a degree of the polynomial.
$\Box$

%Let the polynomial (\ref{poly}) have a bi-degree $(d_1,d_2)$ and a total degree $d$.

%Let us define the polynomial:
%\begin{eqnarray}
%Q(x,y)=P(\alpha_1 x^{k_1}+\ldots+\alpha_m x^{k_m},\alpha_1 y^{k_1}+\ldots+\alpha_m y^{k_m})    
%\end{eqnarray}
%and assume that this polynomial is irreducible. 

%\begin{lemma}
%Bound of the number of $(X,Y)\in \mathbf{X}\times \mathbf{X}$ such that
%\begin{eqnarray}\label{equa-c}
%P(X,Y)=c.
%\end{eqnarray}
%\end{lemma}

\subsection{The number of solutions of a polynomial equation in a subgroup}

There are several known bounds of the number of solutions of a polynomial equation in a subgroup. We use the estimates of Theorem 2 from \cite{MV} and Theorem 1.2 from \cite{KSV}, since they differ from the bound of Corvaja and Zannier (see \cite{C-Z}) only in a multiplicative constant, and our estimates depend on the degree of the polynomial, and not on the genus of its algebraic curve.

For a bi-variate absolutely irreducible polynomial (irreducible over $\overline{\mathbb{F}}_p$)
\begin{equation}
\label{eq:P}
P(X,Y)=\sum_{i+j\le  d} a_{ij}X^{i}Y^{j} \in \mathbb{F}_p[X,Y]
\end{equation}
of bi-degree $(d_1,d_2)$ and total degree $\deg P= d$ ($d_1+d_2\geqslant d$), we 
define $P^{\sharp}(X,Y)$ as the homogeneous polynomial  of degree 
$d^{\sharp}=\min\{ i+j:~a_{ij}\ne 0\}$ 
given by \begin{equation}
\label{eq:Pmin}
P^{\sharp}(X,Y)=\sum_{i+j=d^{\sharp}} a_{ij}X^{i}Y^{j}.
\end{equation}
%Let us call the irreducible polynomial (\ref{eq:P}) 

Let us formulate Theorem 2 from \cite{MV} and Theorem 1.2 from \cite{KSV} in the case $h=1$. We also substitute $d_1+d_2$ instead of $g$ ($g$ is less than $d_1+d_2$ anyway).

%Let us denote by $g_1$ the following constant:
%\begin{eqnarray}\label{g-gcd}
%g_1={\rm gcd}\{ j_1-j_2\mid \exists i_1,i_2\, :\, a_{i_1j_1}a_{i_2j_2}\not= 0 \}.
%\end{eqnarray}
%and by $g_2$ is the constant 
%\begin{equation}
%\label{eq:gk}
%g_2=\gcd\{ i_1+j_1-i_2 -j_2 ~:~ a_{i_1,j_1}  a_{i_2,j_2}  \ne 0\}.  
%\end{equation}
%It is easy to see that $g_1, g_2\le d$. 

%Let us introduce the set of polynomials $\mathcal{P}$:
%$$
%\mathcal{P}=\{ P_{q',q''}(X,Y)\mid P_{q',q''}=P(q'X,q''Y),\, q',q''\in\mathbb{F}_p^*\}
%$$
%and the subset
%$$
%P_k(x,y)=P(q_k'x,q_k''y),\quad k=1,\ldots,h.
%$$
%We call these polynomials 
%$G$-independent if for any integers $1\leqslant i<j\leqslant h$ ratios $q_i'/q_j'$ and $q_i''/q_j''$ do not belong to $G$ simultaneously. 

Let us put %by definition 
\begin{eqnarray}\label{M2}
\mathcal{N}=\{ (x,y)\in G\times G \mid P(x,y)=0\}.
\end{eqnarray}
%In other words, $\mathcal{N}_h$ is the set of solutions $(x,y)\in\bigcup_{k=1}^hG_k^1\times G_k^2$ of equation (\ref{equa-c}), where $G_k^1=q_k'G$, $G_k^2=q_k''G$.

The following two theorems generalize the estimate of Corvaja and Zannier (see~\cite{C-Z}). %and references therein. 

\begin{theorem}[see \cite{MV}]\label{th1}
Consider the following assumptions:
\begin{itemize}
  \item $P(x,y)\in\overline{\mathbb{F}}_p[x,y]$ is an absolutely irreducible polynomial (\ref{eq:P}) having bidegree $(d_1,d_2)$ such that $P(0,0)\not=0$ and $\deg_x P(x,0)\geqslant 1$, $d_2\geqslant 1$;
%    \item polynomials
%$P_1,\ldots,P_h\in\mathcal{P}$ are $G$-independent;
    \item $G$ is a subgroup of $\mathbb{F}_{p^s}^*$ such that 
  $10^3<|G|<\frac{1}{3}p^{3/4}$. %where $h<(40d_1d_2^2)^{-3}|G|^2$.
\end{itemize}
Then the following bound
\begin{eqnarray}\label{est-M1}
\#\mathcal{N}\leqslant 12d_1d_2(d_1+d_2)^2|G|^{2/3} %16mn^2(m+n)|G|^{2/3}
\end{eqnarray}
holds.
\end{theorem}

\begin{theorem}[see \cite{KSV}]\label{th2}
Consider the following assumptions:
\begin{itemize}
  \item $P(x,y)\in\overline{\mathbb{F}}_p[x,y]$ is an absolutely irreducible polynomial (\ref{eq:P}) having bidegree $(d_1,d_2)$ and $P^{\sharp}(X,Y)$ consists of at least two monomials;
%    \item polynomials $P_1,\ldots,P_h\in\mathcal{P}$ are $G$-independent;
    \item $G$ is a subgroup of $\mathbb{F}_{p^s}^*$ such that
$$
c_0(d_1,d_2)\leqslant |G|\leqslant \frac{1}{2}p^{3/4},
$$ 
with a constant $c_0(d_1,d_2)$, depending only on $d_1$ and $d_2$
\end{itemize}
Then the following bound
\begin{eqnarray}\label{est-M1}
\#\mathcal{N}\leqslant 12d_1d_2(d_1+d_2)^2|G|^{2/3} %16mn^2(m+n)|G|^{2/3}
\end{eqnarray}
holds.
%Suppose that $P$ is an absolutely irreducible polynomial,
%$$
%\deg_X P = d_1 \text{ and } \deg_Y P = d_2
%$$
%and also  that $P^{\sharp}(X,Y)$ consists of at least two monomials. 
%There exists a constant $c_0(m,n)$, depending only on $m$ and $n$, 
%such that for any  multiplicative  subgroup  $\mathcal{G} \subseteq \mathbb{F}_p^*$ satisfying 
%$$
%\frac{1}{2}p^{3/4}h^{-1/4}\ge  |G|\ge \max\{h^2, c_0(m,n)\},
%$$ 
% and 
% $\mathcal{G}$-independent polynomials~\eqref{eq:Pk}  we have
% $$
%\sum_{i=1}^h\#\left\{ (u,v)\in \mathcal{G}^2 ~:~P_i(u,v)=0\right\} <12mn(m+n) g h^{2/3}|G|^{2/3}.
%$$
\end{theorem}

\section{Bounds of a polynomial doubling}

Let $a_{ij}$ be coefficients of the polynomial (\ref{eq:P}) and let $k_1\geqslant\ldots\geqslant k_m$. Let us put 
$$
l_1=\min \{ ik_m\mid \exists j\, :\,a_{ij}\not=0,\, n=1,\ldots,m\}\cup \{ 0\},
$$
$$ 
l_2=\min \{ jk_m\mid \exists i\, :\,a_{ij}\not=0,\, n=1,\ldots,m\}\cup \{ 0\}.
$$
We consider the polynomials
\begin{eqnarray}\label{Poly-Q}
Q(x,y)=x^{-l_1}y^{-l_2}P(\alpha_1 x^{k_1}+\ldots+\alpha_m x^{k_m},\beta_1 y^{k_1}+\ldots+\beta_m y^{k_m})    
\end{eqnarray}
where $\alpha_1\cdot\ldots\cdot\alpha_m\beta_1\cdot\ldots\cdot\beta_m\not=0$ and
$$
Q_r(x,y)=Q(x,y)-rx^{-l_1}y^{-l_2}.
$$

\begin{proposition}\label{Qr0}
Let the polynomial $Q_r(x,y)$ be irreducible, the subgroup $G$ be such that $10^3<|G|<\frac{1}{3}p^{3/4}$ and at least one of the two conditions:
\begin{itemize}
\item $Q_r(0,0)\not=0$ and $\deg_x Q_r(x,0)\geqslant 1$ or $\deg_x Q_r(0,y)\geqslant 1$;
\item polynomial ${Q_r}^{\sharp}(X,Y)$ consists of at least two monomials
\end{itemize}
is satisfied. Then the number of solutions of the equation
$$
P(X,Y)=r,\quad X,Y\in \mathbf{X}
$$
does not exceed 
$$
12(d_1k_1-l_1)(d_2k_1-l_2)(d_1k_1+d_2k_1-l_1-l_2)^2(\max(k_1,k_1-k_m))^{2/3}|\mathbf{X}|^{2/3}.
$$
\end{proposition}

{\it Proof.} Since elements $X_n$ and $Y_l$ of the recurrent sequences $\mathbf{X}$ and $\mathbf{Y}$ are represented as
$$
X_n=\alpha_1 x^{k_1}+\ldots+\alpha_m x^{k_m},\quad x=\mu^n\in G,
$$
$$
Y_l=\beta_1 y^{k_1}+\ldots+\beta_m y^{k_m},\quad y=\mu^l\in G,
$$
where $x,y$ are elements of the group $G$. If the first condition is satisfied, then Theorem \ref{th1} can be applied to the polynomial $Q_r(x,y)$ and the subgroup $G$. Theorem \ref{th1} gives us the bound:
$$
\# \{ (x,y)\in G\times G\mid Q_r(x,y)=0 \}\leqslant 12(d_1k_1-l_1)(d_2k_1-l_2)(d_1k_1+d_2k_1-l_1-l_2)^2|G|^{2/3}
$$
because $m=d_1k_1-l_1$, $n=d_2k_1-l_2$, $|G|\leqslant \max(k_1,k_1-k_m)|\mathbf{X}|$.

If the second condition is satisfied, then Theorem \ref{th2} can be applied to the polynomial $Q_r(x,y)$ and the subgroup $G$. Theorem \ref{th2} gives us the bound:
$$
\# \{ (x,y)\in G\times G\mid Q_r(x,y)=0 \}\leqslant 12(d_1k_1-l_1)(d_2k_1-l_2)(d_1k_1+d_2k_1-l_1-l_2)^2|G|^{2/3}
$$
because $m=d_1k_1-l_1$, $n=d_2k_1-l_2$, $|G|\leqslant  \max(k_1,k_1-k_m)|\mathbf{X}|$. $\Box$

%We have $|G|\leqslant |\mathbf{X}|$ because for each $x\in G$ there is $X=\alpha_1 x^{k_1}+\ldots+\alpha_m x^{k_m}\in \mathbf{X}$.

\begin{proposition}\label{PXX}
If the polynomials $P$ and $Q_r$ with $r\in\mathbb{F}_p\setminus\{ \tilde{r}_1,\ldots,\tilde{r}_l\}$ satisfy to condition of Proposition \ref{Qr0} then we have the bound 
$$
P(\mathbf{X},\mathbf{X})=\#\{ P(X,Y)\mid X,Y\in\mathbf{X} \}>\Theta|\mathbf{X}|^{4/3}
$$
where $\Theta$ depends only on $d$ and $l$.
%$C=(12(d_1k_1-l_1)(d_2k_1-l_2)(d_1k_1+d_2k_1-l_1-l_2)^2 \max^{2/3}(k_1,k_1-k_m) )^{-1}$.
\end{proposition}

{\it Proof.} The number of solutions of the equation
$$
Q_r(x,y)=0
$$
is not greater than $C|\mathbf{X}|^{2/3}$ for all $r\in \mathbb{F}_p\setminus\{\tilde{r}_1,\ldots,\tilde{r}_l \}$. The number of pairs $(x,y)\in \mathbf{X}\times \mathbf{X}$ is equal to $|\mathbf{X}|^2$. Consider all such $r_1,\ldots,r_M\in\mathbb{F}_p$ that there exists a pair $(x,y)\in \mathbf{X}\times \mathbf{X}$ such that $Q_{r_j}(x,y)=0$. Any pair $(x,y)\in\mathbf{X}\times \mathbf{X}$ satisfies one of the following equations:
\begin{eqnarray}\label{Qrs}
Q_{r_j}(x,y)=0, \quad j=1,\ldots,M,    
\end{eqnarray}
but the number of solutions for all these equations is not greater than $ld_1|\mathbf{X}|+(M-l)C|\mathbf{X}|^{2/3}\geqslant |\mathbf{X}|^2$. We obtain the bound for the number $M$ of equations (\ref{Qrs}):
$$
M>C^{-1}(|\mathbf{X}|^{4/3}-ld_1|\mathbf{X}|^{1/3})+l>\Theta |\mathbf{X}|^{4/3},
$$
where $C=(12(d_1k_1-l_1)(d_2k_1-l_2)(d_1k_1+d_2k_1-l_1-l_2)^2 \max^{2/3}(k_1,k_1-k_m) )^{-1}$ and $\Theta$ depend only on $d$ and $l$.
$\Box$

To prove the next two lemmas, we will use Newton polygons.
A Newton polygon of the polynomial (\ref{eq:P}) is a polygon in the Cartesian plane that is the convex hull of the points $(i,j)$ for all $a_{ij}\not= 0$. The product of polynomials corresponds to a Newton polygon, which is the Minkowski sum of the Newton polygons of its factors. The Newton polygons of polynomials are located entirely in the first quadrant of the Cartesian plane.

\begin{lemma}\label{P1-irr}
The polynomial 
$$
P_1(x,y)=\alpha x^2 y+\gamma xy^2-rxy+\delta x+\beta y 
$$
with $\alpha\beta\gamma\delta r\not=0$ is irreducible.
\end{lemma}

{\it Proof.} The figure 

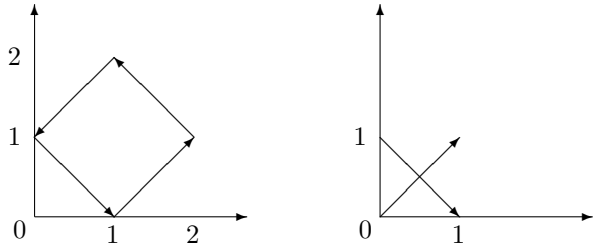
\begin{picture}(120,80)
\put(0,-30){
\begin{picture}(0,400)
\put(-8,-8){$0$}
\put(0,0){\vector(0,1){80}}\put(-10,27){$1$}\put(-10,57){$2$}
\put(0,0){\vector(1,0){80}}\put(27,-10){$1$}\put(57,-10){$2$}
\put(30,0){\vector(1,1){30}}
\put(60,30){\vector(-1,1){30}}
\put(30,60){\vector(-1,-1){30}}
\put(0,30){\vector(1,-1){30}}
\end{picture}}
\put(130,-30){
\begin{picture}(0,400)
\put(-8,-8){$0$}
\put(0,0){\vector(0,1){80}}\put(-10,27){$1$}
\put(0,0){\vector(1,0){80}}\put(27,-10){$1$}
\put(0,0){\vector(1,1){30}}
\put(0,30){\vector(1,-1){30}}
\end{picture}}
\end{picture}
\bigskip\bigskip\bigskip\bigskip\\
on the left hand side shows the Newton polygon of the polynomial $P_1(x,y)$. The Newton polygons of polynomials are located entirely in the first quadrant of the Cartesian plane.

In our case, the square, which is the Newton polygon of the polynomial $P_1(x,y)$, can be represented either as the Minkowski sum of itself with the point $(0,0)$, or as the Minkowski sum of two vectors in the right hand side figure. The first case is degenerate and corresponds to constant factoring. This leaves only the second case, which we will check explicitly. 

Let us suppose that
$$
P_1(x,y)=\alpha x^2 y+\gamma xy^2-rxy+\delta x+\beta y=(Ax+By)(D+Exy). 
$$
It is easy to see that we have not term $-rxy$ on the right hand side of the equality above. That is a contradiction. $\Box$

\begin{lemma}\label{P2-irr}
The polynomial 
$$
P_2(x,y)=xy\left(\left( \alpha x+\frac{\beta}{x}\right)\left(\gamma y+\frac{\delta}{y}\right)-r\right)=\alpha\gamma x^2y^2+\alpha\delta x^2+\beta\gamma y^2-rxy+\beta\delta
$$
with $\alpha\beta\gamma\delta r\not=0$, $r^2\not=4\alpha\beta\gamma\delta$ is irreducible.
\end{lemma}

{\it Proof.} Consider the following five figures.

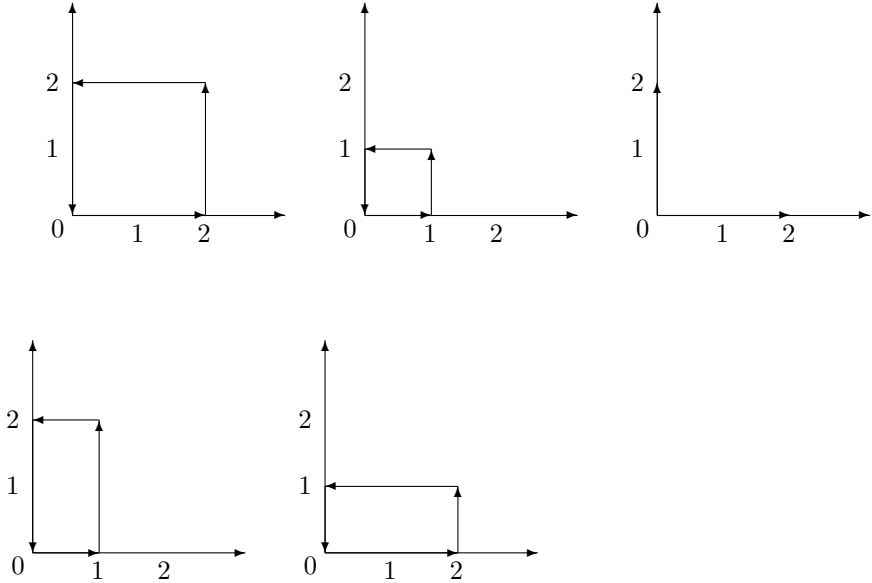
\begin{picture}(120,80)
\put(0,-30){
\begin{picture}(0,400)
\put(-8,-8){$0$}
\put(0,0){\vector(0,1){80}}\put(-10,22){$1$}\put(-10,47){$2$}
\put(0,0){\vector(1,0){80}}\put(22,-10){$1$}\put(47,-10){$2$}
\put(0,0){\vector(1,0){50}}
\put(50,0){\vector(0,1){50}}
\put(50,50){\vector(-1,0){50}}
\put(0,50){\vector(0,-1){50}}
\end{picture}}
\put(110,-30){
\begin{picture}(0,400)%
\put(-8,-8){$0$}
\put(0,0){\vector(0,1){80}}\put(-10,22){$1$}\put(-10,47){$2$}
\put(0,0){\vector(1,0){80}}\put(22,-10){$1$}\put(47,-10){$2$}
\put(0,0){\vector(1,0){25}}
\put(25,0){\vector(0,1){25}}
\put(25,25){\vector(-1,0){25}}
\put(0,25){\vector(0,-1){25}}
\end{picture}}
\put(220,-30){
\begin{picture}(0,400)%
\put(-8,-8){$0$}
\put(0,0){\vector(0,1){80}}\put(-10,22){$1$}\put(-10,47){$2$}
\put(0,0){\vector(1,0){80}}\put(22,-10){$1$}\put(47,-10){$2$}
\put(0,0){\vector(1,0){50}}
\put(0,0){\vector(0,1){50}}
\end{picture}}
\end{picture}
\bigskip\bigskip\bigskip\bigskip\\
\begin{picture}(120,80)
\put(0,-30){
\begin{picture}(0,400)
\put(-8,-8){$0$}
\put(0,0){\vector(0,1){80}}\put(-10,22){$1$}\put(-10,47){$2$}
\put(0,0){\vector(1,0){80}}\put(22,-10){$1$}\put(47,-10){$2$}
\put(0,0){\vector(1,0){25}}
\put(25,0){\vector(0,1){50}}
\put(25,50){\vector(-1,0){25}}
\put(0,50){\vector(0,-1){50}}
\end{picture}}
\put(110,-30){
\begin{picture}(0,400)%
\put(-8,-8){$0$}
\put(0,0){\vector(0,1){80}}\put(-10,22){$1$}\put(-10,47){$2$}
\put(0,0){\vector(1,0){80}}\put(22,-10){$1$}\put(47,-10){$2$}
\put(0,0){\vector(1,0){50}}
\put(50,0){\vector(0,1){25}}
\put(50,25){\vector(-1,0){50}}
\put(0,25){\vector(0,-1){25}}
\end{picture}}
\end{picture}
\bigskip\bigskip\bigskip\bigskip\\
The square on the first graphic with a sides of the length two on shows the Newton polygon of the polynomial $P_2(x,y)$. The square, which is the Newton polygon of the polynomial $P_2(x,y)$, can be represented either as the Minkowski sum of itself with the point $(0,0)$, or as the Minkowski sum of two unit squares in the second figure or as the Minkowski sum of two vectors in the third figure or as a sum of a rectangle and unit vector on forth and fifth figures. The first case is degenerate and corresponds to constant factoring. This leaves only the second, third, forth and fifth cases, which we will check explicitly. 

Let us suppose that
$$
P_2(x,y)=\alpha\gamma x^2y^2+\alpha\delta x^2+\beta\gamma y^2-rxy+\beta\delta=
$$
$$
=(A+Bx+Dy+Exy)(F+Gx+Hy+Kxy).
$$
Without loss of generality, we suppose that $A=F$, otherwise we can multiply the first bracket by $\sqrt{\frac{F}{A}}$ and the second bracket by $\sqrt{\frac{A}{F}}$ $\left(\sqrt{\frac{A}{F}},\sqrt{\frac{F}{A}}\in\overline{\mathbb{F}}_p\right)$. If $A=F$ then $B+G=D+H=0$ and $EG+BK=B(K-E)=0$ as coefficients of $x,y$ and $xy$. Thus we have that
$$
P_2(x,y)=(A+Bx+Dy+Exy)(A-Bx-Dy+Exy).
$$
By coefficients we obtain that
$$
A^2=\beta\delta,\quad E^2=\alpha\gamma,\quad 2AE=-r
$$
and that
$$
4A^2E^2=r^2=4\alpha\beta\gamma\delta.
$$
That is a contradiction.

It remains to study the case shown in the third figure. 
Let us suppose that
$$
P_2(x,y)=(A+Bx+Dx^2)(E+Fy+Gy^2).
$$
It is known that $ADEG\not= 0$, $BF=-r\not=0$ and all coefficients are non-zero. The coefficient of $x$ on the one hand is equal to $BE$, but on the other hand it is equal to zero. 

Let us consider the last two cases. The case
$$
P_2(x,y)=(A+Bx+Dy+Exy+Dy^2+Fxy^2)(G+Fx)
$$
cannot be realized because the polynomial $f(y)=P(-\frac{F}{G},y)\equiv 0$, but it is impossible because $P_2(x,y)$ has the term $\beta\gamma y^2$. The case
$$
P_2(x,y)=(A+Bx+Dy+Exy+Dx^2+Fx^2y)(G+Fy).
$$
cannot be realized because $g(x)=P(x,-\frac{G}{F})\equiv 0$, but it is impossible because $P_2(x,y)$ has the term $\alpha\beta y^2$.
That is a contradiction, and the lemma is proved. $\Box$

\bigskip

{\it Proof of Theorem \ref{th-est}.}
$K$-Fibonacci sequence:
\begin{eqnarray}
F_{n+2}=KF_{n}+F_{n+1},\quad F_0=0,\,\, F_1=1,\quad n=0,1,2,3,\ldots    
\end{eqnarray}
can be represented as the union of two recurrent subsequences $\mathcal{F}_p'$ and $\mathcal{F}_p''$:
$$
F_0,\ldots,F_{2k},\ldots
$$
with $F_k'=F_{2k}$ and
$$
F_1,\ldots,F_{2k+1},\ldots
$$
with $F_k''=F_{2k+1}$, and $\mathcal{F}_p=\mathcal{F}_p'\cup\mathcal{F}_p''$, because the $K$-Fibonacci sequence is the union:
$$
F_1',F_1'',F_2',F_2'',\ldots,F_k',F_k'',\ldots
$$ 
Sequences $\mathcal{F}_p'$ and $\mathcal{F}_p''$ are both recurrent and
$$
F_{k+2}'=(K^2+2)F_{k+1}'-F_k',\quad F_1'=0,\, F_2'=K;
$$ 
$$
F_{k+2}''=(K^2+2)F_{k+1}''-F_k'',\quad F_1''=1,\, F_2''=K^2+1.
$$
We obtain the lower bounds of cardinalities of sets $\mathcal{F}_p'+\mathcal{F}_p'$ and $\mathcal{F}_p'\cdot\mathcal{F}_p'$ and use that:
$$
\mathcal{F}_p+\mathcal{F}_p\supset \mathcal{F}_p'+\mathcal{F}_p'\Rightarrow |\mathcal{F}_p+\mathcal{F}_p|\geqslant |\mathcal{F}_p'+\mathcal{F}_p'|;
$$
$$
\mathcal{F}_p\cdot\mathcal{F}_p\supset (\mathcal{F}_p'\cdot\mathcal{F}_p') \Rightarrow |\mathcal{F}_p\cdot\mathcal{F}_p|\geqslant |\mathcal{F}_p'\cdot\mathcal{F}_p'|.
$$
Characteristic equation of the sequence $\mathcal{F}_p'$ is $\lambda^2-(K^2+1)\lambda+1=0$, 
$$
\lambda_1=\lambda_2^{-1}=\frac{-(K^2+1)+\sqrt{(K^2+1)^2-4}}{2}.
$$ 

To estimate $|\mathcal{F}_p'+\mathcal{F}_p'|$ we apply Proposition \ref{PXX} with polynomial $P(x,y)=x+y$. By Lemma \ref{P1-irr} we obtain that polynomials $Q_r(x,y)=P_1(x,y)$ are irreducible for any $r\not= 0$. By the bound of Proposition \ref{PXX} we obtain the estimate:
$$
|\mathcal{F}_p+\mathcal{F}_p|\geqslant |\mathcal{F}_p'+\mathcal{F}_p'|>\Theta |\mathcal{F}_p'|^{4/3}=\frac{\Theta}{2\sqrt[3]{2}} |\mathcal{F}_p|^{4/3}
$$
To estimate $|\mathcal{F}_p'\cdot\mathcal{F}_p'|$ we apply Proposition \ref{PXX} with polynomial $P(x,y)=xy$. We obtain by Lemma \ref{P2-irr} that polynomials $Q_r(x,y)=P_2(x,y)$ are irreducible for any $r\not= 0,\pm 2\sqrt{\alpha\beta\gamma\delta}$. By the bound of Proposition  \ref{PXX} we obtain the estimate:
$$
|\mathcal{F}_p\cdot\mathcal{F}_p|\geqslant |\mathcal{F}_p'\cdot\mathcal{F}_p'|>\Theta |\mathcal{F}_p'|^{4/3}=\frac{\Theta}{2\sqrt[3]{2}} |\mathcal{F}_p|^{4/3}.\,\, \Box
$$

\section{Acknowledges}

The authors are grateful to Alexander Derevtsov for reading the manuscript and identifying typos.

Vyugin I.V.\\
HSE University,\\
{\it ilyavyugin@yandex.ru}\\
\\
Dutta S.\\
HSE University,\\
{\it duttamathe@gmail.com}\\


\begin{thebibliography}{99}

\bibitem{AV}
Aleshina, S.\,A., V’yugin, I.\,V. On a Polynomial Version of the Sum-Product Problem for Subgroups // Math Notes 113, 3-9 (2023). https://doi.org/10.1134/S0001434623010017


\bibitem{B-L}
Benfield, B., Lippard, O. {\em Connecting Zeros in Pisano Periods to Prime Factors of K-Fibonacci Numbers} // The Fibonacci Quarterly, 63(2), 240-258 (2025). https://doi.org/10.1080/00150517.2025.2460555

\bibitem{C-Z}
P. Corvaja, U. Zannier, {\em Greatest common divisor of $u-1$, $v-1$ in positive characteristic and rational points on curves over finite fields} //
J. Eur. Math. Soc., 15:5, 1927-1942, 2013.

\bibitem{KSV}
Konyagin S.\,V., Shparlinski I.\,E., Vyugin I.\,V., Polynomial Equations in Subgroups and Applications // {\it In: A. Avila, M. Th. Rassias, Y. Sinai (eds.), Analysis at Large, Dedicated to the Life and Work of Jean Bourgain, Springer, 2022, pp. 273-297.}

\bibitem{MV}
S. Makarychev, I. Vyugin, {\em Solutions of Polynomial Equations in Subgroups of $\mathbb{F}_p$} //
Arnold Math J. 5, 105-121 (2019).


\bibitem{PS}
Parmanand Singh, {\em The so-called Fibonacci numbers in ancient and medieval India} // Historia Mathematica, V. 12, Issue 3, August 1985, Pages 229-244.





\end{thebibliography}
\end{document}